\documentclass[11pt,oneside,english]{amsart}
\usepackage[T1]{fontenc}
\usepackage[latin9]{inputenc}
\usepackage{units}
\usepackage{amsthm}
\usepackage{amstext}
\usepackage{amssymb}
\usepackage{esint}
\makeatletter
\numberwithin{equation}{section}
\numberwithin{figure}{section}
  \theoremstyle{plain}
  \newtheorem*{thm*}{\protect\theoremname}
\theoremstyle{plain}
\newtheorem{thm}{\protect\theoremname}
  \theoremstyle{plain}
  \newtheorem{lem}[thm]{\protect\lemmaname}

\makeatother

\usepackage{babel}
  \providecommand{\lemmaname}{Lemma}
  \providecommand{\theoremname}{Theorem}
\providecommand{\theoremname}{Theorem}

\begin{document}

\title{Hydrodynamic limit of the Boltzmann-Monge-Ampere system}

\author{Fethi Ben Belgacem}

\address{LR03ES04 Equations aux dérivées partielles et applications, Université
de Tunis El Manar, Faculté des Sciences de Tunis, 2092 Tunis, Tunisie.}

\email{fethi.benbelgacem@isimm.rnu.tn}

\begin{abstract}
In this paper we investigate the hydrodynamic limit of the Boltzmann-Monge-Ampere
system in the so-called quasineutral regime. We prove the convergence
of the Boltzmann-Monge-Ampere system to the Euler equation by using
the relative entropy method.
\end{abstract}
\maketitle

\section{Introduction and main results}

The goal of this article is to study the hydrodynamical limit of the
Boltzman-Monge-Ampere system (BMA)
\begin{align}
\partial_{t}f^{\varepsilon}+\xi.\nabla_{x}f^{\varepsilon}+\nabla_{x}\varphi^{\varepsilon}.\nabla_{\xi}f^{\varepsilon}= & Q(f^{\varepsilon},f^{\varepsilon}),\label{eq:B.M}\\
\mbox{det}\left(\mathbb{I}_{d}+\varepsilon^{2}D^{2}\varphi^{\varepsilon}\right)= & \rho^{\varepsilon},\label{eq:M.A}
\end{align}
 where $\mathbb{I}_{d}$ is the identity matrix and 
\begin{equation}
\rho^{\varepsilon}(t,x)=\int_{\mathbb{R}^{d}}f^{\varepsilon}(t,x,\xi)d\xi\label{eq:densit=0000E9-1-1}
\end{equation}
and $f^{\varepsilon}(t,x,\xi)\geq0$ is the electronic density at
time $t\geq0$ point $x\in\left[0,1\right]^{d}=\mathbb{T}^{d},$ and
with a velocity $\xi\in\mathbb{R}^{d}.$ The periodic electric potential
$\varphi^{\varepsilon}$ is coupled with $f^{\varepsilon}$ through
the nonlinear Monge-Ampere equation (\ref{eq:M.A}). The quantities
$\varepsilon>0$ and $Q(f^{\varepsilon},f^{\varepsilon})$ denote
respectively the vacum electric permitivity and the Boltzman collision
integral. This latter, is given by (see {[}3,9{]}) 
\[
Q(f^{\varepsilon},f^{\varepsilon})(t,x,\xi)=\int\int_{S_{+}^{d-1}\times\mathbb{R}^{d}}\left(\left(f^{\epsilon}\right)^{\prime}\left(f_{1}^{\varepsilon}\right)^{\prime}-f^{\varepsilon}f_{1}^{\varepsilon}\right)b\left(\xi-\xi_{1},\sigma\right)d\sigma d\xi_{1},
\]

where the terms $f_{1}^{\varepsilon},$ $\left(f^{\varepsilon}\right)^{\prime}$
and $\left(f_{1}^{\varepsilon}\right)^{\prime}$ defines, respectively
the values $f^{\varepsilon}(t,x,\xi_{1}),$ $f^{\varepsilon}(t,x,\xi^{\prime})$
and $f^{\varepsilon}(t,x,\xi_{1}^{\prime})$ with $\xi^{\prime}$
and $\xi_{1}^{\prime}$ given in terms of $\xi$, $\xi_{1}\in\mathbb{R}^{d},$
and $\sigma\in S_{+}^{d-1}=\left\{ \sigma\in S^{d-1}/\sigma.\xi\geq\sigma.\xi_{1}\right\} $
by 
\[
\xi^{\prime}=\frac{\xi+\xi_{1}}{2}+\frac{\xi-\xi_{1}}{2}\sigma,\:\:\xi_{1}^{\prime}=\frac{\xi+\xi_{1}}{2}-\frac{\xi-\xi_{1}}{2}\sigma.
\]

By linearising the determinant about the identity matrix $\mathbb{I}_{d},$
one get 
\[
\mbox{det}\left(\mathbb{I}_{d}+\varepsilon^{2}D^{2}\varphi^{\varepsilon}\right)=1+\varepsilon^{2}\triangle\varphi^{\varepsilon}+O\left(\varepsilon^{4}\right).
\]

It follows that the BMA system is a fully nonlinear version of the
Vlasov-Poisson-Boltzman (VPB) system defined by 
\begin{eqnarray}
\partial_{t}f^{\varepsilon}+\xi.\nabla_{x}f^{\varepsilon}+\nabla_{x}\varphi^{\varepsilon}.\nabla_{\xi}f^{\varepsilon} & = & Q(f^{\varepsilon},f^{\varepsilon})\label{eq:B-P-1}\\
\varepsilon^{2}\triangle\varphi^{\varepsilon} & = & \rho^{\varepsilon}-1\label{eq:B-P2-1}
\end{eqnarray}

This latter, has been interested many authors. In {[}5{]} DiPerna
and Lions showed the existence of renormalized solution. Desvilletes
and Dolbeault {[}7{]} are interested to the long-time behavior of
the weak solutions of the VPB system for the initial boundary problem.
In {[}10{]} Guo established the global existence of smooth solutions
to the VPB system in periodic boundary condition case. For more references
for this subject, Boltzmann equation or Vlasov\textendash{}Poisson
system, one can see {[}1-7, 9\textendash{}14{]}. 

In {[}11{]} L. Hsiao and al. studied the convergence of the VPB system
to the Incompressible Euler Equations. If one consider the case $Q(f^{\varepsilon},f^{\varepsilon})=0$,
we obtain the Vlasov-Monge-Ampère(VMA). This problem, was been considered
by Y. Bernier and Grégoire{[}2{]}. They showed that weak solution
of VMA converge to a solution of the incompressible Euler equations
when the parameter $\varepsilon$ goes to $0.$ 

This work aims to extend these efforts to study such systems.

First, Note that 
\[
\int_{\mathbb{R}^{d}}Q(f^{\varepsilon},f^{\varepsilon})d\xi=\int_{\mathbb{R}^{d}}\xi_{i}Q(f^{\varepsilon},f^{\varepsilon})d\xi=\int_{\mathbb{R}^{d}}\left|\xi\right|^{2}Q(f^{\varepsilon},f^{\varepsilon})d\xi=0,\mbox{ }i=1,2,...,d.
\]
and the conservation of total energy 
\begin{equation}
\frac{1}{2}\int_{\mathbb{R}^{d}}\int_{\mathbb{T}^{d}}\left|\xi\right|^{2}f^{\varepsilon}\left(t,x,\xi\right)dxd\xi+\frac{\varepsilon}{2}\int_{\mathbb{T}^{d}}\left|\nabla\phi^{\varepsilon}\left(t,x\right)\right|^{2}dx=E_{0}\label{eq:energy}
\end{equation}
where 
\[
J^{\varepsilon}\left(t,x\right)=\int_{\mathbb{R}^{d}}\xi f^{\varepsilon}\left(t,x,\xi\right)d\xi.
\]
From the conservation of mass and momentum, it follows that 
\begin{equation}
\partial_{t}\rho^{\varepsilon}+\nabla.J^{\varepsilon}=0\label{eq:ro}
\end{equation}
and 
\begin{equation}
\partial_{t}J^{\varepsilon}+\nabla_{x}.\int_{\mathbb{R}^{d}}\left(\xi\otimes\xi\right)f^{\varepsilon}d\xi+\nabla\phi^{\varepsilon}+\frac{\varepsilon}{2}\nabla\left(\left|\nabla\phi^{\varepsilon}\right|^{2}\right)-\varepsilon\nabla.\left(\nabla\phi^{\varepsilon}\otimes\nabla\phi^{\varepsilon}\right)=0.\label{eq:J}
\end{equation}

Let us consider the periodic boundary problem of Euler equations to
the incompressible fluid 
\begin{alignat}{1}
\nabla.u=0,\: t>0,\: x\in\mathbb{T}^{d}\label{eq:E1}\\
\partial_{t}u+\left(u.\nabla\right)u+\nabla p=0 & \: t>0,\: x\in\mathbb{T}^{d}\label{eq:E2}\\
u\left(0,x\right)=u_{0}\left(x\right)\in\mathcal{H}^{s},\label{eq:E3}
\end{alignat}
where the function space $\mathcal{H}^{s}$ is given by $\mathcal{H}^{s}=\left\{ u\in H^{s}\left(\mathbb{T}^{d}\right),\:\nabla.u=0\right\} .$

We have the following result.
\begin{thm*}
Let $0<T<T^{*}$ and $u_{0}$ in $\mathcal{H}^{s}\left(s>1+\frac{d}{2}\right),$
$\mathbb{Z}^{d}$ periodic in $x$. Assume that $f_{0}^{\varepsilon}\left(x,\xi\right)\geq0$
to be smooth, $\mathbb{Z}^{d}$ periodic in $x$, and $f_{0}^{\varepsilon}$
decays fast as $\xi\rightarrow\infty.$ In addition, we assume that
\[
\int_{\mathbb{R}^{d}}f_{0}^{\varepsilon}\left(x,\xi\right)d\xi=1+o\left(\varepsilon^{\frac{1}{2}}\right),\textrm{ as }\varepsilon\rightarrow0,
\]
in the strong sense of the space $H^{-1}\left(\mathbb{T}^{d}\right)$
and 
\[
\frac{1}{2}\int_{\mathbb{R}^{d}}\int_{\mathbb{T}^{d}}\left|\xi-u_{0}\left(x\right)\right|^{2}f_{0}^{\varepsilon}\left(x,\xi\right)dxd\xi+\frac{\varepsilon}{2}\int_{\mathbb{T}^{d}}\left|\nabla\phi^{2}\left(0,x\right)\right|^{2}dx\rightarrow0\textrm{ as }\varepsilon\rightarrow0.
\]
 Let $f^{\varepsilon}$ be any nonnegative smooth solution of (\ref{eq:B.M})-(\ref{eq:M.A}).
Then, up to the extraction of a subsequence, the current $J^{\varepsilon}$
converges weakly to the unique solution $u\left(x,t\right)$ of the
Euler equations (\ref{eq:E1})-(\ref{eq:E2})-(\ref{eq:E3}). Moreover,
the divergence free part of $f$ converges to $u$ in $L^{\infty}\left(\left[0,T\right],L^{2}\left(\mathbb{T}^{d}\right)\right).$ 
\end{thm*}

\section{Proof of the theorem }

First introduce the modulated energy functional
\[
H^{\varepsilon}\left(t\right)=\frac{1}{2}\int_{\mathbb{R}^{d}}\int_{\mathbb{T}^{d}}\left|\xi-u\left(x\right)\right|^{2}f^{\varepsilon}\left(t,x,\xi\right)dxd\xi+\frac{\varepsilon}{2}\int_{\mathbb{T}^{d}}\left|\nabla\phi^{2}\left(t,x\right)\right|^{2}dx.
\]

In the squel we need the following two Lemmas
\begin{lem}
Under the hypothesis of the above theorem ,we have up to the extraction
of a sequence, $\rho^{\varepsilon}$ converges to 1 in $C^{0}\left(\left[0,T\right],\mathcal{D}^{\prime}\left(\mathbb{T}^{d}\right)\right),$
the current $J^{\varepsilon}$ converges to $J$ in $L^{\infty}\left(\left[0,T\right],\mathcal{D}^{\prime}\left(\mathbb{T}^{d}\right)\right),$
$J\in L^{\infty}\left(\left[0,T\right],L^{2}\left(\mathbb{T}^{d}\right)\right),$
and the divergence free parts of $J^{\varepsilon}$ converges to $J$
in $C^{0}\left(\left[0,T\right],\mathcal{D}^{\prime}\left(\mathbb{T}^{d}\right)\right).$\end{lem}
\begin{proof}
we take $d=2,$ and we notice that 
\[
\det\left(I+\varepsilon D^{2}\phi^{\varepsilon}\right)=1+\varepsilon\triangle\phi^{\varepsilon}+\varepsilon^{2}\det D^{2}\phi^{\varepsilon}.
\]
We first show that $\rho^{\varepsilon}\rightarrow1$ in $C^{0}\left(\left[0,T\right],\mathcal{D}^{\prime}\left(\mathbb{T}^{d}\right)\right).$
In fact, for $\eta\in C_{0}^{\infty}\left(\mathbb{T}^{d}\right),$
we get
\begin{align*}
\int\left(\rho^{\varepsilon}\left(t,x\right)-1\right)\eta\left(x\right)dx & =\int\left(\det\left(I+\varepsilon D^{2}\phi^{\varepsilon}\right)-1\right)\eta\left(x\right)dx\\
 & =\int\left(\varepsilon\triangle\phi^{\varepsilon}+\varepsilon^{2}\det D^{2}\phi^{\varepsilon}\right)\eta\left(x\right)dx.
\end{align*}
But 
\[
\det D^{2}\phi^{\varepsilon}=\frac{1}{2}\textrm{tr}\left(\left(\textrm{cof}D^{2}\phi^{\varepsilon}\right)D^{2}\phi^{\varepsilon}\right)=\frac{1}{2}\textrm{div}\left(\left(\textrm{cof}D^{2}\phi^{\varepsilon}\right)\nabla\phi^{\varepsilon}\right),
\]
it follows by integrating by parts that
\begin{alignat*}{1}
\int\left(\rho^{\varepsilon}\left(t,x\right)-1\right)\eta\left(x\right)dx & =\varepsilon\int\nabla\phi^{\varepsilon}\nabla\eta\left(x\right)dx+\frac{\varepsilon^{2}}{2}\int\textrm{div}\left(\left(\textrm{cof}D^{2}\phi^{\varepsilon}\right)\nabla\phi^{\varepsilon}\right)\eta\left(x\right)dx\\
 & =\varepsilon\int\nabla\phi^{\varepsilon}\nabla\eta\left(x\right)dx-\frac{\varepsilon^{2}}{2}\int\left(\textrm{cof}D^{2}\phi^{\varepsilon}\right)\nabla\phi^{\varepsilon}.\nabla\eta\left(x\right)dx.
\end{alignat*}
Thus, by the Hölder inequality one has 
\begin{alignat*}{1}
\left|\int\left(\rho^{\varepsilon}\left(t,x\right)-1\right)\right|\eta\left(x\right)dx & \leq\varepsilon^{\nicefrac{1}{2}}\left(\varepsilon\int\left|\nabla\phi^{\varepsilon}\right|^{2}\right)^{\nicefrac{1}{2}}\left(\int\left|\nabla\eta\right|^{2}\right)^{\nicefrac{1}{2}}+\\
 & +\frac{\varepsilon^{2}}{2}\left\Vert \textrm{cof}D^{2}\phi^{\varepsilon}\right\Vert _{L^{2}}\left\Vert \nabla\phi^{\varepsilon}\right\Vert _{L^{2}}\left\Vert \nabla\eta\right\Vert _{L^{2}}.
\end{alignat*}
Recall that from regularity result of Monge-Ampère equation we have
{[}8{]} 
\[
\left\Vert \textrm{cof}D^{2}\phi^{\varepsilon}\right\Vert _{L^{2}}\lesssim\varepsilon^{-\frac{1}{2}},
\]
So, by the conservation of the energy, one deduce that
\begin{alignat*}{1}
\left|\int\left(\rho^{\varepsilon}\left(t,x\right)-1\right)\eta\left(x\right)dx\right| & \leq C_{0}\varepsilon^{\nicefrac{1}{2}}\left\Vert \nabla\eta\right\Vert _{L^{2}}+C.\varepsilon^{\nicefrac{3}{2}}\left\Vert \nabla\phi^{\varepsilon}\right\Vert _{L^{2}}\left\Vert \nabla\eta\right\Vert _{L^{2}}.\\
 & \leq C_{0}\varepsilon^{\nicefrac{1}{2}}\left\Vert \nabla\eta\right\Vert _{L^{2}}+C.\varepsilon\left\Vert \nabla\eta\right\Vert _{L^{2}}\\
 & \leq\varepsilon^{\nicefrac{1}{2}}\left(C_{0}+C\varepsilon^{\nicefrac{1}{2}}\right)\left\Vert \nabla\eta\right\Vert _{L^{2}}
\end{alignat*}

\end{proof}
By the total energy equality (\ref{eq:energy}) we have 
\begin{equation}
\int\left|J^{\varepsilon}\left(t,x\right)\right|dx\leq\left(\int\int\left|\xi\right|^{2}f^{\varepsilon}\left(t,x,\xi\right)dxd\xi\right)^{\nicefrac{1}{2}}\left(\int\int f^{\varepsilon}\left(t,x,\xi\right)dxd\xi\right)^{\nicefrac{1}{2}}\leq C.\label{eq:inegality-1}
\end{equation}
 Thus $J^{\varepsilon}$ is bounded in $L^{\infty}\left(\left[0,T\right],L^{1}\left(\mathbb{T}^{d}\right)\right).$
Up to extracting a subsequence, we can assume that $J^{\varepsilon}$
has a limit $J$ in the sens of (Radon) measures on $\left[0,T\right]\times\nicefrac{\mathbb{R}^{d}}{\mathbb{Z}^{d}}=\mathbb{T}^{d}.$
Let us define as in {[}11{]}, for each non-negative function $z\left(t\right)\in C^{0}\left(\left[0,T\right]\right),$
the convex functional of a (Radon) measure
\begin{alignat*}{1}
K\left(\rho^{\varepsilon},J^{\varepsilon}\right) & =\int_{0}^{T}\frac{\left|J^{\varepsilon}\left(t,x\right)\right|^{2}}{2\rho^{\varepsilon}\left(t,x\right)}z\left(t\right)dxdt\\
 & =\sup_{b}\int_{0}^{T}\left\{ -\frac{1}{2}\left|b\left(t,x\right)\right|^{2}\rho^{\varepsilon}\left(t,x\right)+b\left(t,x\right)J^{\varepsilon}\left(t,x\right)\right\} z\left(t\right)dt.
\end{alignat*}
where $b$ belongs to the space of all continuous functions from $\left[0,T\right]\times\mathbb{T}^{d}$
to $\mathbb{R}^{d}.$ From (\ref{eq:inegality-1}) and since the functional
$K$ is lower semi-continuous with respect to the convergence of measure,
it follows that 
\[
\int_{0}^{T}z\left(t\right)\left(\int\left|J\left(t,x\right)\right|^{2}dx\right)dt\leq C\int_{0}^{T}z\left(t\right)dt,
\]
which means that $J\in L^{\infty}\left(\left[0,T\right],L^{2}\left(\mathbb{T}^{d}\right)\right).$

From (\ref{eq:ro}) and (\ref{eq:J}) one write 
\[
\partial_{t}\rho^{\varepsilon}=\partial_{t}\mbox{det}\left(\mathbb{I}+\varepsilon^{2}D^{2}\varphi^{\varepsilon}\right)=-\nabla J^{\varepsilon},
\]
 thus 
\[
\nabla J^{\varepsilon}=-\varepsilon\partial_{t}\triangle\phi^{\varepsilon}-\varepsilon^{2}\partial_{t}\det D^{2}\phi^{\varepsilon}.
\]
For $\eta\in C_{0}^{\infty}\left(\mathbb{T}^{d}\right),$ we have
\[
\int\nabla J^{\varepsilon}\eta\left(x\right)dx=-\varepsilon\int\partial_{t}\left(\triangle\phi^{\varepsilon}\eta\right)dx-\varepsilon^{2}\int\partial_{t}\det D^{2}\phi^{\varepsilon}\eta dx,
\]
thus $J$ is divergence free in $x$ in the sense of distribution.

By (\ref{eq:J}), we deduce that $\partial_{t}J$ is bounded in $L^{\infty}\left(\left[0,T\right],D^{\prime}\left(\mathbb{T}^{d}\right)\right).$
So, we obtain that up to the exraction of a subsequance, $J\in C^{0}\left(\left[0,T\right],L^{2}\left(\mathbb{T}^{d}\right)-w\right).$ 

In the same way , we can show that the divergence -free part of $J^{\varepsilon}$
converges to $J$ in $C^{0}\left(\left[0,T\right],D^{\prime}\left(\mathbb{T}^{d}\right)\right).$
Since $J^{\varepsilon}$ converges to $J$, it remains to show that
$J=u$ in $L^{\infty}\left(\left[0,T\right],L^{2}\left(\mathbb{T}^{d}\right)\right).$
For this, it suffies to use the next Lemma.
\begin{lem}
{[}11{]}Let $u$ be the unique solution of the Euler equations (\ref{eq:E1})-(\ref{eq:E2})
with initial datum and $u_{0}$ and the hypotheses of theorem 1 hold.
Then, for any $t\in\left(0,T\right],$ $H^{\varepsilon}\left(t\right)\rightarrow0$
as $\varepsilon\rightarrow0.$
\end{lem}
To end the proof of the Theorem, we define a new functional 
\begin{equation}
h^{\varepsilon}\left(t\right)=\int\frac{\left|J^{\varepsilon}\left(t,x\right)-\rho^{\varepsilon}\left(t,x\right)u\left(t,x\right)\right|^{2}}{2\rho^{\varepsilon}\left(t,x)\right)}dx.\label{eq:hepsilon-1}
\end{equation}
With $b$ belongs to the space of all continuous functions from $\mathbb{T}^{d}$
to $\mathbb{R}^{d}.$ By the Cauchy-Shwarz inequality, one get 
\[
h^{\varepsilon}\left(t\right)\leq\frac{1}{2}\int\left|\xi-u\left(t,x\right)\right|^{2}f^{\varepsilon}\left(t,x,\xi\right)dxd\xi\leq H^{\varepsilon}\left(t\right).
\]
Since $\rho^{\varepsilon}\rightarrow1,$ $J^{\varepsilon}\rightarrow J$
and from the convexity of the functional defined by (\ref{eq:hepsilon-1}),
we obtain 
\[
\int\left|J\left(t,x\right)-u\left(t,x\right)\right|^{2}dx\leq2\lim_{\varepsilon\rightarrow0}h^{\varepsilon}\left(t\right)\leq2\lim_{\varepsilon\rightarrow0}H^{\varepsilon}\left(t\right)=0.
\]

This finish the proof of Theorem 1.2.

\end{document}